\documentclass[12pt,a4paper]{article}
\usepackage{mathrsfs}
\usepackage{epsfig, graphicx}
\usepackage{latexsym,amsfonts,amsbsy,amssymb}
\usepackage{amsmath,amsthm}
\makeatletter 

\@addtoreset{equation}{section}
\makeatother 
\allowdisplaybreaks 

\newtheorem{theorem}{Theorem}[section]
\newtheorem{lemma}{Lemma}[section]

\newtheorem{remark}{Remark}[section]
\newtheorem{algorithm}{Algorithm}[section]

\begin{document}
\title{A Full Multigrid Method for Eigenvalue Problems\footnote{This work is supported in part
by the National Science Foundation of China (NSFC 91330202, 11371026,
11001259, 11031006, 2011CB309703),  the National
Center for Mathematics and Interdisciplinary Science, CAS and the
President Foundation of AMSS-CAS.}}
\author{Hehu Xie\footnote{LSEC, NCMIS, Institute
of Computational Mathematics, Academy of Mathematics and Systems
Science, Chinese Academy of Sciences, Beijing 100190,
China(hhxie@lsec.cc.ac.cn)} }
\date{}
\maketitle
\begin{abstract}
In this paper, a full (nested) multigrid scheme is proposed
to solve eigenvalue problems. The idea here is to use the multilevel correction method
to transform the solution of eigenvalue problem to a series of solutions of the corresponding
boundary value problems and eigenvalue problems defined on the coarsest finite element space.
The boundary value problems which are define on a sequence of multilevel finite element space
can be solved by some multigrid iteration steps.
Besides the multigrid iteration, all other efficient iteration methods for solving boundary value
problems can serve as linear problem solver. The computational work of this new scheme can reach
optimal order the same as solving the corresponding source problem.
Therefore, this type of iteration scheme improves the efficiency of eigenvalue problem solving.


\end{abstract}

\section{Introduction}

It is well known there have existed many efficient algorithms, such as multigrid method
 and many other precondition techniques \cite{BrennerScott,Shaidurov,Xu},
 for solving boundary value problems. 
The error bounds of the approximate solution
obtained from these efficient numerical algorithms are comparable to the theoretical
bounds determined by the finite element discretization.  But the amount of computational
work involved is only proportional to the number of unknowns
in the discretized equations. For more details of the multigrid and multilevel methods,
please refer to \cite{BankDupont,Bramble,BramblePasciak,BrambleZhang,BrennerScott,
Hackbusch,Hackbusch_Book,McCormick,ScottZhang,Shaidurov,Xu,Xu_Two_Grid} and the references cited therein.

But there is no many efficient numerical methods for solving eigenvalue problems with optimal
complexity. Solving large scale eigenvalue problems is one of  fundamental problems in modern science
and engineering society. However, it is always a very difficult task to solve high-dimensional
eigenvalue problems which come from physical and chemistry sciences.
Recently, a type of multilevel correction method is proposed for solving
eigenvalue problems in \cite{LinXie_MultiLevel,Xie_IMA,Xie_JCP}. In this multilevel
correction scheme, the solution of eigenvalue problem on the final level
mesh can be reduced to a series of solutions of boundary value problems on
the multilevel meshes and a series of solutions of the eigenvalue problem
 on the coarsest mesh. The multilevel correction method gives a way to construct
the multigrid method for eigenvalue problems \cite{Xie_IMA,Xie_JCP}.

The aim of this paper is to present a full multigrid method
for solving eigenvalue problems based on the combination of the multilevel correction
method \cite{Xie_IMA,Xie_JCP} and the multigrid iteration for
boundary value problems. Comparing with the method in \cite{LinXie_MultiLevel,Xie_IMA,Xie_JCP},
the difference is that we do not solve the linear boundary value problem
exactly in each correction step with the multigrid method. We only
get an approximate solution with some multigrid iteration steps. In this new version of
multigrid method,  solving eigenvalue
problem will not be much more difficult than the multigrid scheme for the
corresponding boundary value problems.
It is worth to noting that besides the multigrid method here,
other types of numerical algorithms such as BPX multilevel preconditioners \cite{Xu},
algebraic multigrid method and domain decomposition
preconditioners (cf. \cite{BrennerScott,ToselliWidlund})
can also act as the linear algebraic solvers for boundary value problems.


An outline of the paper goes as follows. In Section 2, we introduce the
finite element method for eigenvalue problem and the corresponding basic
error estimates. A type of full multigrid
algorithm for solving eigenvalue problem by finite element method
 is given in Section 3.
Two numerical examples are presented to validate our
theoretical analysis in section 4.
Some concluding remarks are given in the last section.

\section{Finite element method for eigenvalue problem}\label{sec;preliminary}
This section is devoted to introducing some notation and the finite element
method for eigenvalue problem. In this paper, we shall use the standard notation
for Sobolev spaces $W^{s,p}(\Omega)$ and their
associated norms and semi-norms (cf. \cite{Adams}). For $p=2$, we denote
$H^s(\Omega)=W^{s,2}(\Omega)$ and
$H_0^1(\Omega)=\{v\in H^1(\Omega):\ v|_{\partial\Omega}=0\}$,
where $v|_{\Omega}=0$ is in the sense of trace,
$\|\cdot\|_{s,\Omega}=\|\cdot\|_{s,2,\Omega}$.
The letter $C$ (with or without subscripts) denotes a generic
positive constant which may be different at its different occurrences
 through the paper.

For simplicity, we consider the following model problem to illustrate the main idea:
Find $(\lambda, u)$ such that
\begin{equation}\label{LaplaceEigenProblem}
\left\{
\begin{array}{rcl}
-\nabla\cdot(\mathcal{A}\nabla u)+\phi u&=&\lambda u, \quad {\rm in} \  \Omega,\\
u&=&0, \ \  \quad {\rm on}\  \partial\Omega,
\end{array}
\right.
\end{equation}
where $\mathcal{A}$ is a symmetric and positive definite matrix with suitable
regularity, $\phi$ is a nonnegative function,  $\Omega\subset\mathcal{R}^d$ $(d=2,3)$
 is a bounded domain with
Lipschitz boundary $\partial\Omega$ and $\nabla$, $\nabla\cdot$
denote the gradient, divergence operators, respectively.

In order to use the finite element method to solve
the eigenvalue problem (\ref{LaplaceEigenProblem}), we need to define
the corresponding variational form as follows:
Find $(\lambda, u )\in \mathcal{R}\times V$ such that $b(u,u)=1$ and
\begin{eqnarray}\label{weak_eigenvalue_problem}
a(u,v)&=&\lambda b(u,v),\quad \forall v\in V,
\end{eqnarray}
where $V:=H_0^1(\Omega)$ and
\begin{equation}\label{inner_product_a_b}
a(u,v)=\int_{\Omega}\big(\nabla u\cdot\mathcal{A}\nabla v +\phi uv\big)d\Omega,
 \ \ \ \  \ \ b(u,v) = \int_{\Omega}uv d\Omega.
\end{equation}
The norms $\|\cdot\|_a$ and $\|\cdot\|_b$ are defined by
\begin{eqnarray*}
\|v\|_a=a(v,v)^{1/2}\ \ \ \ \ {\rm and}\ \ \ \ \ \|v\|_b=b(v,v)^{1/2}.
\end{eqnarray*}
  It is well known that the eigenvalue problem (\ref{weak_eigenvalue_problem})
  has an eigenvalue sequence $\{\lambda_j \}$ (cf. \cite{BabuskaOsborn_Book,Chatelin}):
$$0<\lambda_1\leq \lambda_2\leq\cdots\leq\lambda_k\leq\cdots,\ \ \
\lim_{k\rightarrow\infty}\lambda_k=\infty,$$ and associated
eigenfunctions
$$u_1, u_2, \cdots, u_k, \cdots,$$
where $b(u_i,u_j)=\delta_{ij}$ ($\delta_{ij}$ denotes the Kronecker function).
In the sequence $\{\lambda_j\}$, the $\lambda_j$ are repeated according to their
geometric multiplicity.

Now, let us define the finite element approximations of the problem
(\ref{weak_eigenvalue_problem}). First we generate a shape-regular
decomposition of the computing domain $\Omega\subset \mathcal{R}^d\
(d=2,3)$ into triangles or rectangles for $d=2$ (tetrahedrons or
hexahedrons for $d=3$) (cf. \cite{BrennerScott,Ciarlet}).
The diameter of a cell $K\in\mathcal{T}_h$ is denoted by $h_K$ and
the mesh size $h$ describes  the maximum diameter of all cells
$K\in\mathcal{T}_h$. Based on the mesh $\mathcal{T}_h$, we can
construct a finite element space denoted by $V_h \subset V$.
For simplicity, we set $V_h$ as the linear finite element space
which is defined as follows
\begin{equation}\label{linear_fe_space}
V_h = \big\{ v_h \in C(\Omega)\ \big|\ v_h|_{K} \in \mathcal{P}_1,
\ \ \forall K \in \mathcal{T}_h\big\},
\end{equation}
where $\mathcal{P}_1$ denotes the linear function space.



The standard finite element scheme for eigenvalue
 problem (\ref{weak_eigenvalue_problem}) is:
Find $(\bar{\lambda}_h, \bar{u}_h)\in \mathcal{R}\times V_h$
such that $b(\bar{u}_h,\bar{u}_h)=1$ and
\begin{eqnarray}\label{Weak_Eigenvalue_Discrete}
a(\bar{u}_h,v_h)
&=&\bar{\lambda}_h b(\bar{u}_h,v_h),\quad\ \  \ \forall v_h\in V_h.
\end{eqnarray}
From \cite{BabuskaOsborn_1989,BabuskaOsborn_Book,Chatelin}, 
the discrete eigenvalue problem (\ref{Weak_Eigenvalue_Discrete}) has eigenvalues:
$$0<\bar{\lambda}_{1,h}\leq \bar{\lambda}_{2,h}\leq\cdots\leq \bar{\lambda}_{k,h}
\leq\cdots\leq \bar{\lambda}_{N_h,h},$$
and corresponding eigenfunctions
$$\bar{u}_{1,h}, \bar{u}_{2,h},\cdots, \bar{u}_{k,h}, \cdots, \bar{u}_{N_h,h},$$
where $b(\bar{u}_{i,h},\bar{u}_{j,h})=\delta_{ij}, 1\leq i,j\leq N_h$ ($N_h$ is
the dimension of the finite element space $V_h$).

Let $M(\lambda_i)$ denote the eigenspace corresponding to the
eigenvalue $\lambda_i$ which is defined by
\begin{eqnarray}
M(\lambda_i)&=&\big\{w\in H_0^1(\Omega): w\ {\rm is\ an\ eigenvalue\ of\
(\ref{weak_eigenvalue_problem})} \nonumber\\
&&\ \ \ \ \ \ \ \ \ \ \ \ \  {\rm corresponding\ to}\ \lambda_i\  {\rm and}\ b(w,w)=1\big\},
\end{eqnarray}
and define
\begin{eqnarray}
\delta_h(\lambda_i)=\sup_{w\in M(\lambda_i)}\inf_{v_h\in
V_h}\|w-v_h\|_{a}.
\end{eqnarray}

Let us define the following quantity:
\begin{eqnarray}
\eta_{a}(h)&=&\sup_{f\in L^2(\Omega),\|f\|_b=1}\inf_{v_h\in V_h}\|Tf-v_h\|_{a},\label{eta_a_h_Def}
\end{eqnarray}
where $T:L^2(\Omega)\rightarrow V$ is defined as
\begin{equation}\label{laplace_source_operator}
  a(Tf,v) = b(f,v), \ \ \ \ \  \forall f \in L^2(\Omega) \ \ \  {\rm and}\  \ \ \forall v\in V.
\end{equation}
Then the error estimates for the eigenpair approximations by the finite
element method can be described as follows.
\begin{lemma}(\cite[Lemma 3.6, Theorem 4.4]{BabuskaOsborn_1989} and \cite{Chatelin})
\label{Err_Eigen_Global_Lem}
For any eigenpair approximation
$(\bar{\lambda}_{i,h},\bar{u}_{i,h})$ $(i = 1, 2, \cdots, N_h)$ of
(\ref{Weak_Eigenvalue_Discrete}),
 there exists an exact eigenpair  $(\lambda_i, u_i)$ of
(\ref{weak_eigenvalue_problem})
such that $b(u_i, u_i) = 1$ and
\begin{eqnarray}
\|u_i-\bar{u}_{i,h}\|_{a}
&\leq& \big(1+C_i\eta_a(h)\big)\delta_h(\lambda_i),\label{Err_Eigenfunction_Global_1_Norm} \\
\|u_i-\bar{u}_{i,h}\|_{b}
&\leq& C_i\eta_{a}(h)\|u_i - u_{i,h}\|_{a},\label{Err_Eigenfunction_Global_0_Norm}\\
|\lambda_i-\bar{\lambda}_{i,h}|
&\leq&C_i\|u_i-\bar{u}_{i,h}\|_{a}^2. \label{Estimate_Eigenvalue}
\end{eqnarray}
Here and hereafter $C_i$ is some constant depending on $i$ but independent of  the mesh size $h$.
\end{lemma}


\section{Full multigrid algorithm for eigenvalue problem}
Recently, a multilevel correction scheme is introduced
in \cite{LinXie_MultiLevel,Xie_IMA,Xie_JCP} for
solving eigenvalue problems. Based on the idea of multilevel correction scheme,
we propose a type of full multigrid method for
eigenvalue problems here. The main idea in this method is to approximate
 the underlying boundary value
problems on each level by some multigrid smoothing iteration steps.
In order to describe the full multigrid method, we first introduce
 the sequence of finite element
spaces. 
We generate a coarse mesh $\mathcal{T}_H$
with the mesh size $H$ and the coarse linear finite element space $V_H$ is
defined on the mesh $\mathcal{T}_H$. Then we define a sequence of
 triangulations $\mathcal{T}_{h_k}$
of $\Omega\subset \mathcal{R}^d$ determined as follows.
Suppose $\mathcal{T}_{h_1}$ (produced from $\mathcal{T}_H$ by
regular refinements) is given and let $\mathcal{T}_{h_k}$ be obtained
from $\mathcal{T}_{h_{k-1}}$ via one regular refinement step
(produce $\beta^d$ subelements) such that
\begin{eqnarray}\label{mesh_size_recur}
h_k=\frac{1}{\beta}h_{k-1},\ \ \ \ k=2,\cdots,n,
\end{eqnarray}
where the positive number $\beta$ denotes the refinement index
 and larger than $1$ (always equals $2$).
Based on this sequence of meshes, we construct the corresponding
 nested linear finite element spaces such that
\begin{eqnarray}\label{FEM_Space_Series}
V_{H}\subseteq V_{h_1}\subset V_{h_2}\subset\cdots\subset V_{h_n}.
\end{eqnarray}
The sequence of finite element spaces
$V_{h_1}\subset V_{h_2}\subset\cdots\subset V_{h_n}$
 and the finite element space $V_H$ have  the following relations
of approximation accuracy
\begin{eqnarray}\label{delta_recur_relation}
\eta_a(H)\gtrsim\delta_{h_1}(\lambda_i),\ \ \ \
\delta_{h_k}(\lambda_i)=\frac{1}{\beta}\delta_{h_{k-1}}(\lambda_i),\ \ \ k=2,\cdots,n.
\end{eqnarray}

\subsection{One correction step}
In order to design the full multigrid method, we introduce an one correction step in this subsection.

Assume we have obtained an eigenpair approximation
$(\lambda_{\ell,h_k},u_{\ell,h_k})\in \mathcal{R}\times V_{h_k}$.
Now we introduce a type of iteration step to improve the accuracy of the
current eigenpair approximation $(\lambda_{\ell,h_k},u_{\ell,h_k})$.

\begin{algorithm}\label{Multigrid_Smoothing_Step} One Correction Step
\begin{enumerate}
\item Define the following auxiliary source problem:
Find $\widehat{u}_{\ell+1,h_k}\in V_{h_k}$ such that
\begin{eqnarray}\label{aux_problem}
a(\widehat{u}_{\ell+1,h_k},v_{h_k})&=&\lambda_{\ell,h_k}
b(u_{\ell,h_k},v_{h_k}),\ \
\ \forall v_{h_k}\in V_{h_k}.
\end{eqnarray}
Perform $m$ multigrid iteration steps with the initial value $u_{\ell,h_k}$
to obtain a new eigenfunction approximation
$\widetilde{u}_{\ell+1,h_k}\in V_{h_k}$ by
\begin{eqnarray}\label{Multigrid_Iteration_Step}
\widetilde{u}_{\ell+1,h_k} = MG(V_{h_k},\lambda_{\ell,h_k}u_{\ell,h_k},u_{\ell,h_k},m),
\end{eqnarray}
where $V_{h_k}$ denotes the working space for the multigrid iteration,
$\lambda_{\ell,h_k}u_{\ell,h_k}$ is the right hand side term of the linear equation,
$u_{\ell,h_k}$ denotes  the initial guess and $m$ is the number of multigrid iteration times.
\item  Define a new finite element
space $V_{H,h_k}=V_H+{\rm span}\{\widetilde{u}_{\ell+1,h_k}\}$ and solve
the following eigenvalue problem:
Find $(\lambda_{\ell+1,h_k},u_{\ell+1,h_k})\in\mathcal{R}\times V_{H,h_k}$ such
that $b(u_{\ell+1,h_k},u_{\ell+1,h_k})=1$ and
\begin{eqnarray}\label{Eigen_Augment_Problem}
a(u_{\ell+1,h_k},v_{H,h_k})&=&\lambda_{\ell+1,h_k} b(u_{\ell+1,h_k},v_{H,h_k}),\ \ \
\forall v_{H,h_k}\in V_{H,h_k}.
\end{eqnarray}
\end{enumerate}
In order to simplify the notation and summarize the above two steps, we define
\begin{eqnarray*}
(\lambda_{\ell+1,h_k},u_{\ell+1,h_k})=EigenMG(V_H,\lambda_{\ell,h_k}, u_{\ell,h_k},V_{h_k},m).
\end{eqnarray*}
\end{algorithm}
\begin{theorem}\label{Error_Estimate_One_Smoothing_Theorem}
Assume the multigrid iteration
 $\widetilde{u}_{\ell+1,h_k} = MG(V_{h_k},\lambda_{\ell,h_k}u_{\ell,h_k},u_{\ell,h_k},m)$
  has the following error reduction rate
\begin{eqnarray}\label{MG_Theta}
\|\widehat{u}_{\ell+1,h_k}-\widetilde{u}_{\ell+1,h_k}\|_a
\leq\theta \|\widehat{u}_{\ell+1,h_k}-u_{\ell,h_k}\|_a,
\end{eqnarray}
and $(\lambda_{\ell,h_k},u_{\ell,h_k})$ has the following properties
\begin{eqnarray}
\|\bar{u}_{h_k}-u_{\ell,h_k}\|_b&\leq&
C_i\eta_a(H)\|\bar{u}_{h_k}-u_{\ell,h_k}\|_a,\label{Estimate_h_k_b}\\
|\bar{\lambda}_{h_k}-\lambda_{\ell,h_k}|&\leq&
C_i\|\bar{u}_{h_k}-u_{\ell,h_k}\|_a^2.\label{Estimate_h_k_Eigenvalue}
\end{eqnarray}
After performing the one correction step defined in
Algorithm \ref{Multigrid_Smoothing_Step},
the resultant eigenpair approximation $(\lambda_{\ell+1,h_k},u_{\ell+1,h_k})\in\mathcal{R}\times V_{h_k}$ has the
following error estimates
\begin{eqnarray}
\|\bar{u}_{h_k}-u_{\ell+1,h_k}\|_a &\leq &
\gamma \|\bar{u}_{h_k}-u_{\ell,h_k}\|_a,\label{Estimate_h_k_1_a}\\
\|\bar{u}_{h_k}-u_{\ell+1,h_k}\|_b&\leq&
C_i\eta_a(H)\|\bar{u}_{h_k}-u_{\ell+1,h_k}\|_a,\label{Estimate_h_k_1_b}\\
|\bar{\lambda}_{h_k}-\lambda_{\ell+1,h_k}|&\leq&
C_i\|\bar{u}_{h_k}-u_{\ell+1,h_k}\|_a^2.\label{Estimate_h_k_1_Eigenvalue}
\end{eqnarray}
where
\begin{eqnarray}\label{Gamma_Definition}
\gamma=\theta+(1+2\theta)C_i\eta_a(H)+(1+\theta)C_i^2\eta_a^2(H).
\end{eqnarray}
\end{theorem}
\begin{proof}
From (\ref{Weak_Eigenvalue_Discrete}) and (\ref{aux_problem}), we have
\begin{eqnarray}\label{One_Correction_1}
a(\bar{u}_{h_k}-\widehat{u}_{\ell+1,h_k},v_{h_k})
=b(\bar{\lambda}_{h_k}\bar{u}_{h_k}-\lambda_{\ell,h_k}u_{\ell,h_k},v_{h_k}),
\ \ \ \ \forall v_{h_k}\in V_{h_k}.
\end{eqnarray}
It leads to the following estimates
\begin{eqnarray}\label{One_Correction_2}
\|\bar{u}_{h_k}-\widehat{u}_{\ell+1,h_k}\|_a &\leq&
\|\bar{\lambda}_{h_k}\bar{u}_{h_k}-\lambda_{\ell,h_k}u_{\ell,h_k}\|_b\nonumber\\
&\leq& |\bar{\lambda}_{h_k}-\lambda_{\ell,h_k}|+\|\bar{u}_{h_k}-u_{\ell,h_k}\|_b\nonumber\\
&\leq& C_i\eta_a(H)\|\bar{u}_{h_k}-u_{\ell,h_k}\|_a.
\end{eqnarray}
Combining (\ref{MG_Theta}) and (\ref{One_Correction_2}) leads to the following
linear solving error estimate for $\widetilde{u}_{\ell+1,h_k}$
\begin{eqnarray}\label{One_Correction_3}
\|\widehat{u}_{\ell+1,h_k}-\widetilde{u}_{\ell+1,h_k}\|_a&\leq&
\theta \|\widehat{u}_{\ell+1,h_k}-u_{\ell,h_k}\|_a\nonumber\\
&\leq&\theta \big(\|\widehat{u}_{\ell+1,h_k}-\bar{u}_{h_k}\|_a
 +\|\bar{u}_{h_k}-u_{\ell,h_k}\|_a\big)\nonumber\\
&\leq& \theta\big(1+C_i\eta_a(H)\big)\|\bar{u}_{h_k}-u_{\ell,h_k}\|_a.
\end{eqnarray}
Then from (\ref{One_Correction_2}) and (\ref{One_Correction_3}), we have the following inequalities
\begin{eqnarray}\label{One_Correction_4}
\|\bar{u}_{h_k}-\widetilde{u}_{\ell+1,h_k}\|_a&\leq& \|\bar{u}_{h_k}-\widehat{u}_{\ell+1,h_k}\|_a
+\|\widehat{u}_{\ell+1,h_k}-\widetilde{u}_{\ell+1,h_k}\|_a\nonumber\\
&\leq& \big(\theta+(1+\theta)C_i\eta_a(H)\big)\|\bar{u}_{h_k}-u_{\ell,h_k}\|_a.
\end{eqnarray}

The eigenvalue problem (\ref{Eigen_Augment_Problem})
 can be regarded  as a finite dimensional subspace approximation of the eigenvalue problem
  (\ref{Weak_Eigenvalue_Discrete}).
Similarly to Lemma \ref{Err_Eigen_Global_Lem} (see \cite[Theorem 4.4]{BabuskaOsborn_1989}),
from the second step in Algorithm \ref{Multigrid_Smoothing_Step} and (\ref{One_Correction_4}),
the following estimates hold
\begin{eqnarray}\label{Error_u_u_h_2}
\|\bar{u}_{h_k}-u_{\ell+1, h_k}\|_a&\leq&
\big(1+C_i\widetilde{\eta}_a(H)\big)\inf_{v_{H,h_k}\in V_{H,h_k}}\|\bar{u}_{h_k}-v_{H,h_k}\|_a\nonumber\\
&\leq& \big(1+C_i\eta_a(H)\big)\|\bar{u}_{h_k}-\widetilde{u}_{\ell+1,h_k}\|_a\nonumber\\
&\leq&\gamma \|\bar{u}_{h_k}-u_{\ell,h_k}\|_a
\end{eqnarray}
and
\begin{eqnarray}\label{Error_u_u_h_2_Negative}
\|\bar{u}_{h_k}-u_{\ell+1,h_k}\|_b&\leq&
C_i\widetilde{\eta}_a(H)\|\bar{u}_{h_k}-u_{\ell+1,h_k}\|_a\nonumber\\
&\leq& C_i\eta_a(H)\|\bar{u}_{h_k}-u_{\ell+1,h_k}\|_a,\\
|\bar{\lambda}_{h_k}-\lambda_{\ell+1,h_k}|&\leq&
C_i\|\bar{u}_{h_k}-u_{\ell+1,h_k}\|_a^2,
\end{eqnarray}
where
\begin{eqnarray}\label{Eta_a_h_2}
\widetilde{\eta}_a(H)&=&\sup_{f\in V,\|f\|_0=1}\inf_{v\in
V_{H,h_k}}\|Tf-v\|_a \leq \eta_a(H).
\end{eqnarray}
Then we obtained the desired results (\ref{Estimate_h_k_1_a})-(\ref{Estimate_h_k_1_Eigenvalue})
and complete the proof.
\end{proof}

\subsection{Full multigrid method for eigenvalue problem}
In this subsection, we introduce a full multigrid
scheme based on the {\it One Correction Step} defined in Algorithm
\ref{Multigrid_Smoothing_Step}. This type of full multigrid method can obtain
the optimal error estimate with the optimal computational work.

Since the multigrid method for the boundary value problem has the uniform
error reduction rate, we can choose suitable $m$ such that $\theta<1$ in (\ref{MG_Theta}).
From (\ref{Gamma_Definition}), we have $\gamma<1$ if $H$ is small enough.
From this observation, we can build the following full multigrid method for solving
eigenvalue problems.
\begin{algorithm}\label{Full_Multigrid}Full Multigrid Scheme
\begin{enumerate}
\item Solve the following eigenvalue problem in $V_{h_1}$:
Find $(\lambda_{h_1}, u_{h_1})\in \mathcal{R}\times V_{h_1}$ such that
\begin{equation*}
a(u_{h_1}, v_{h_1}) = \lambda_{h_1} b(u_{h_1}, v_{h_1}), \quad \forall v_{h_1}\in  V_{h_1}.
\end{equation*}
Solve this eigenvalue problem to get an eigenpair approximation $(\lambda_{h_1},u_{h_1})\in\mathcal{R}\times V_{h_1}$.
\item For $k=2,\cdots,n$, do the following iteration
\begin{itemize}
\item Set $u_{0,h_k} = u_{h_{k-1}}$.
\item Do the following multigrid iteration
\begin{eqnarray*}
(\lambda_{\ell+1,h_k}, u_{\ell+1,h_k})=EigenMG(V_H,\lambda_{\ell,h_k},u_{\ell,h_k},V_{h_k},m),
 \ \ \ \ {\rm for}\ \ell=0,\cdots, p-1.
\end{eqnarray*}
\item set $\lambda_{h_k}=\lambda_{p,h_k}$ and $u_{h_k}=u_{p,h_k}$.
\end{itemize}
end Do
\end{enumerate}
Finally, we obtain an eigenpair approximation
$(\lambda_{h_n},u_{h_n})\in \mathcal{R}\times V_{h_n}$.
\end{algorithm}
\begin{theorem}
After implementing Algorithm \ref{Full_Multigrid}, the resultant
eigenpair approximation $(\lambda_{h_n},u_{h_n})$ has the following
error estimate
\begin{eqnarray}
\|\bar{u}_{h_n}-u_{h_n}\|_a &\leq&
C\frac{\gamma^p}{1-\beta\gamma^p}\delta_{h_n}(\lambda),\label{FM_Err_fun}\\
|\bar{\lambda}_{h_n}-\lambda_{h_n}|&\leq&C\delta_{h_n}^2(\lambda),\label{FM_Err_eigen}
\end{eqnarray}
under the condition $\beta\gamma^p<1$.
\end{theorem}
\begin{proof}
Define $e_{k}:=\bar{u}_{h_k}-u_{h_k}$. Then from step 1
in Algorithm \ref{Full_Multigrid}, it is obvious $e_1=0$. For $k=2,\cdots,n$, we have
\begin{eqnarray}\label{FM_Estimate_1}
\|e_k\|_a&\leq& \gamma^p\|\bar{u}_{h_k}-u_{h_{k-1}}\|_a\nonumber\\
&\leq& \gamma^p\big(\|\bar{u}_{h_k}-\bar{u}_{h_{k-1}}\|_a
+\|\bar{u}_{h_{k-1}}-u_{h_{k-1}}\|_a\big)\nonumber\\
&\leq& \gamma^p \big(C\delta_{h_k}(\lambda)+\|e_{k-1}\|_a\big).
\end{eqnarray}
By iterating inequality (\ref{FM_Estimate_1}) and $\beta\gamma^p<1$, the following
inequalities hold
\begin{eqnarray}
\|e_n\|_a&\leq& C\gamma^p\delta_{h_n}(\lambda)+C\gamma^{2p}\delta_{h_{n-1}}(\lambda)+\cdots
+C\gamma^{(n-1)p}\delta_{h_2}(\lambda)\nonumber\\
&\leq& C\sum_{k=2}^n \gamma^{(n-k+1)p}\delta_{h_{\ell}}(\lambda)
=C\left(\sum_{k=2}^n \big(\beta\gamma^{p}\big)^{n-k}\right)\gamma^{p}\delta_{h_n}(\lambda)\nonumber\\
&\leq&C\frac{\gamma^p}{1-\beta\gamma^p}\delta_{h_n}(\lambda).
\end{eqnarray}
For such choice of $p$, we arrive the desired result (\ref{FM_Err_fun}) and
(\ref{FM_Err_eigen}) can be obtained by (\ref{Estimate_Eigenvalue}) and (\ref{FM_Err_fun}).
\end{proof}
\begin{remark}
The good convergence rate of the multigrid method for boundary value problems leads to that
we do not need to choose large $m$ and $p$ \cite{BrennerScott,Hackbusch_Book,Shaidurov,Xu}.
\end{remark}
Now we turn our attention to the estimate of computational work
for {\it Full Multigrid Scheme \ref{Full_Multigrid}}. We will show that Algorithm \ref{Full_Multigrid}
makes solving eigenvalue problem need almost the same work as solving the corresponding
boundary value problem.

First, we define the dimension of each level
finite element space as $N_k:={\rm dim}V_{h_k}$. Then we have
\begin{eqnarray}\label{relation_dimension}
N_k=\Big(\frac{1}{\beta}\Big)^{d(n-k)}N_n,\ \ \ k=1,2,\cdots, n.
\end{eqnarray}
\begin{theorem}
Assume the eigenvalue problem solved in the coarse spaces $V_{H}$ and $V_{h_1}$ need work
$\mathcal{O}(M_H)$ and $\mathcal{O}(M_{h_1})$, respectively, and the work of the
multigrid solver $MG(V_{h_k},\lambda_{\ell,h_k}u_{\ell,h_k},u_{\ell,h_k},m)$
 in each level space $V_{h_k}$ is $\mathcal{O}(N_k)$ for $k=2,3,\cdots,n$. Then the work involved in the
Full Multigrid Scheme \ref{Full_Multigrid} is  $\mathcal{O}(N_n+M_H\log(N_n)+M_{h_1})$.
Furthermore, the complexity will be $\mathcal{O}(N_n)$ provided $M_H\ll N_k$ and $M_{h_1}\leq N_k$.
\end{theorem}
\begin{proof}
Let $W_{k}$ denote the work in the correction step in the $k$-th finite element space $V_{h_{k}}$.
Then with the correction definition in Algorithm \ref{Multigrid_Smoothing_Step}, we have
\begin{eqnarray}\label{work_k}
W_{k}&=&\mathcal{O}(N_{k}+M_H), \ \ \ \  \ k=2,\cdots, n.
\end{eqnarray}
Iterating (\ref{work_k}) and using the fact (\ref{relation_dimension}), we obtain
\begin{eqnarray*}
{\rm Total\ work}&=&\sum_{k=1}^nW_k=\mathcal{O}\Big(M_{h_1}+\sum_{k=2}^n\big(N_{k}+M_H\big)\Big)\nonumber\\
&=&\mathcal{O}\Big(M_{h_1}+(n-1)M_H+\sum_{k=2}^n\big(\frac{1}{\beta}\big)^{d(n-k)}N_n\Big)\nonumber\\
&=&\mathcal{O}(N_n+M_H\log N_n+M_{h_1}).
\end{eqnarray*}
This is the desired result $\mathcal{O}(N_n+M_H\log N_n+M_{h_1})$ and the
one $\mathcal{O}(N_n)$  can be obtained by the conditions $M_H\ll N_n$ and $M_{h_1}\leq N_n$.
\end{proof}

\section{Numerical results}
In this section, two numerical examples are presented to illustrate the
efficiency of the full multigrid scheme proposed in this
paper.

\subsection{Model eigenvalue problem}
Here we give the numerical results of the full multigrid
scheme for the model eigenvalue problem:
Find $(\lambda,u)$ such that
\begin{equation}\label{problem}
\left\{
\begin{array}{rcl}
-\Delta u&=&\lambda u,\quad{\rm in}\ \Omega,\\
u&=&0,\quad\ \  {\rm on}\ \partial\Omega,\\
\int_{\Omega} u^2d\Omega&=&1,
\end{array}
\right.
\end{equation}
where $\Omega=(0,1)\times(0,1)$.

The sequence of finite element spaces are constructed by
using linear element on the series of meshes which are produced by
regular refinement with $\beta =2$ (connecting the midpoints of each edge).
In this example, we use two meshes which are generated by Delaunay method as the initial mesh
$\mathcal{T}_{h_1}$ and set $\mathcal{T}_H=\mathcal{T}_{h_1}$
to investigate the convergence behaviors.
Figure \ref{Initial_Mesh} shows the corresponding
initial meshes: one is coarse and the other is fine.

Algorithm \ref{Full_Multigrid} is applied to solve the eigenvalue problem. In this subsection, we choose
$m=2$ and $2$ conjugate gradient smoothing steps for the presmoothing and postsmoothing in each
multigrid iteration step in Algorithm \ref{Multigrid_Smoothing_Step}.  In each level of the
full multigrid scheme defined in Algorithm \ref{Full_Multigrid}, we only do $2$ multigrid
iteration steps ($p=2$) defined in Algorithm \ref{Multigrid_Smoothing_Step}. 
For comparison, we also solve the eigenvalue problem
by the direct method.
\begin{figure}[ht]
\centering
\includegraphics[width=5.5cm,height=5.5cm]{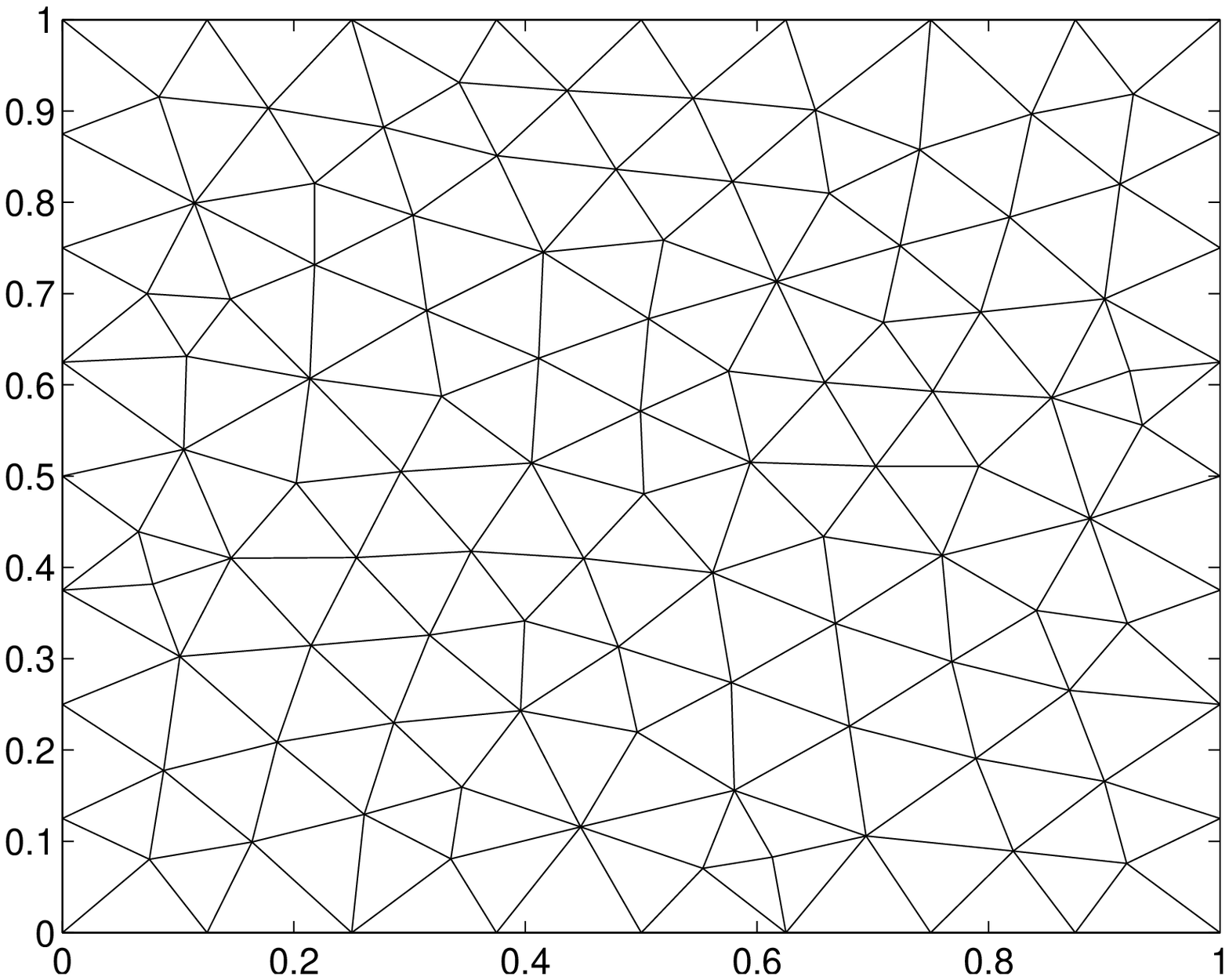}
\includegraphics[width=5.5cm,height=5.5cm]{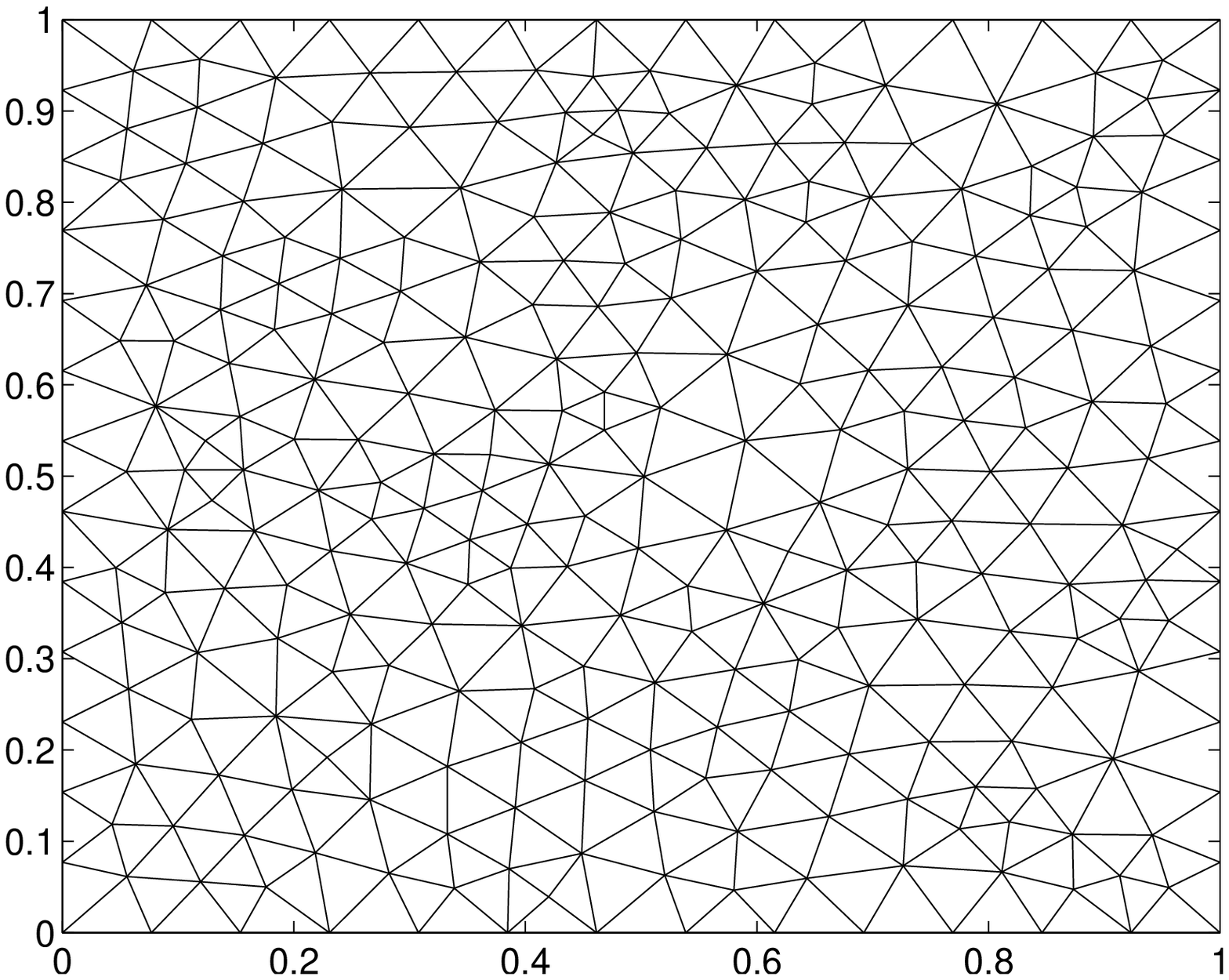}
\caption{\small\texttt The coarse and fine initial meshes
for the unit square} \label{Initial_Mesh}
\end{figure}
Figure \ref{numerical_multi_grid_Model}
gives the corresponding numerical results for the first eigenvalue
$\lambda_1=2\pi^2$ and the corresponding eigenfunction on the two
initial meshes illustrated in Figure \ref{Initial_Mesh}.
\begin{figure}[ht]
\centering
\includegraphics[width=7cm,height=6cm]{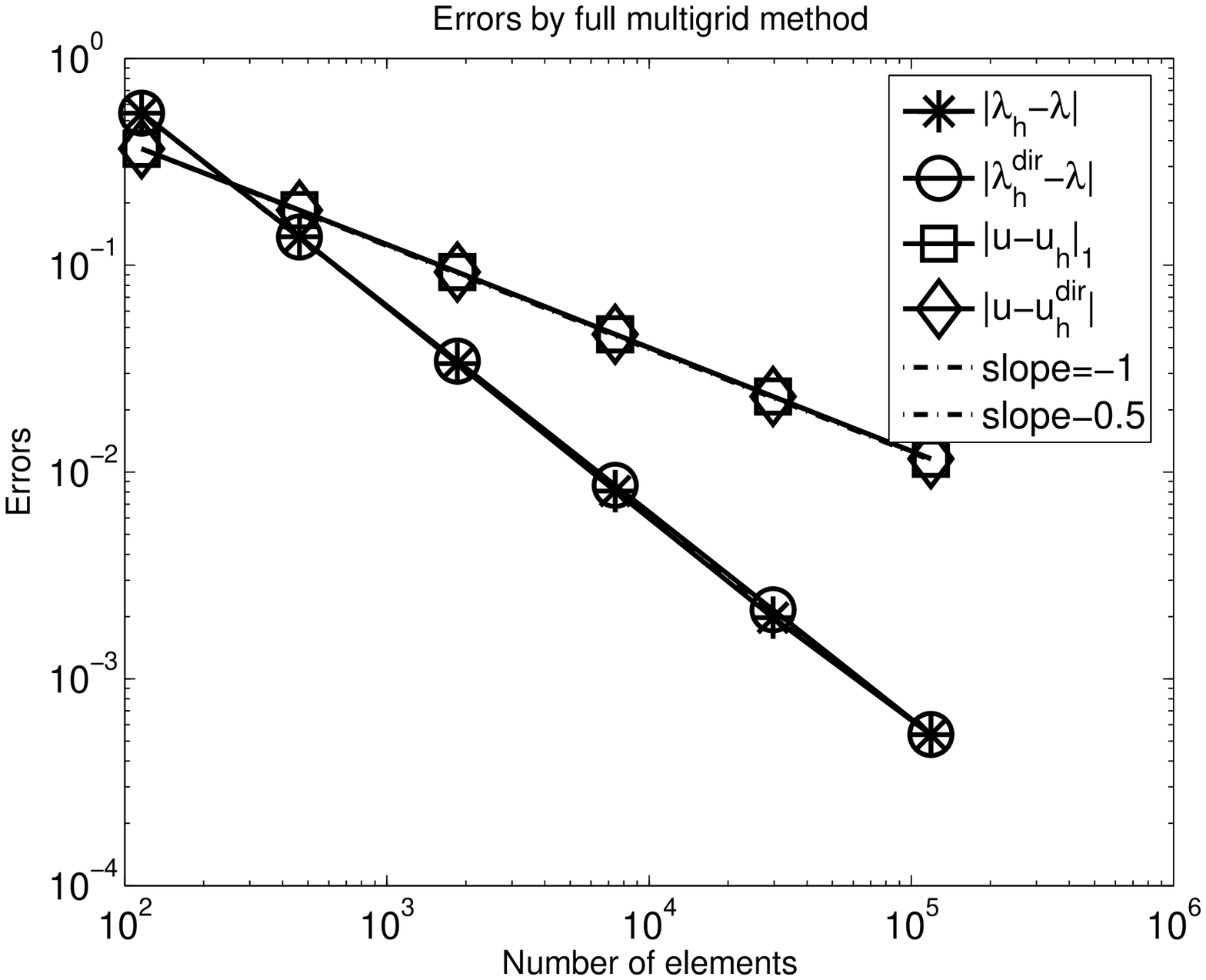}
\includegraphics[width=7cm,height=6cm]{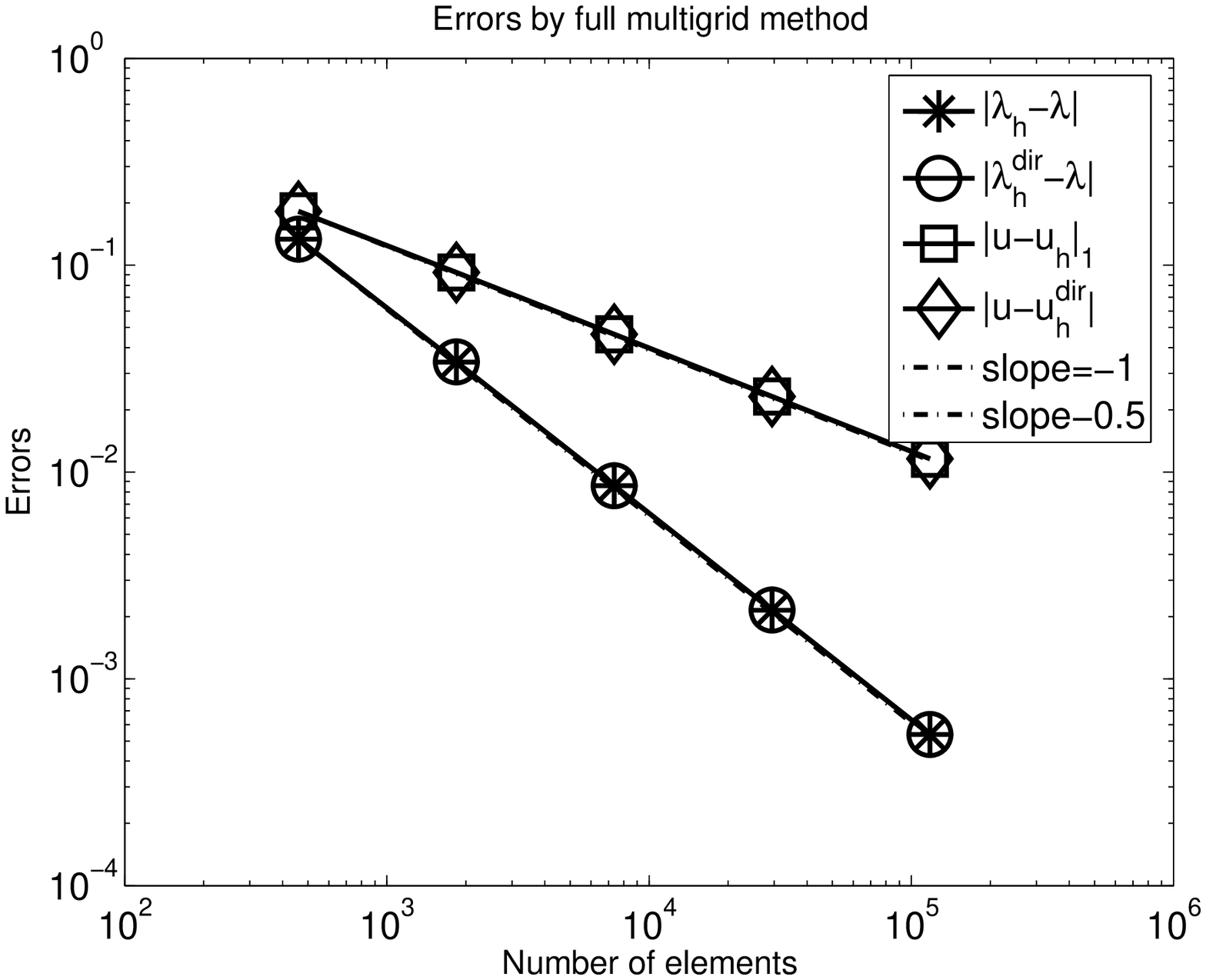}
\caption{\small\texttt The errors of the full multigrid
algorithm for the first six eigenvalues on the unit square, where $u_h$
 and $\lambda_h$ denote the eigenfunction and eigenvalue approximations 
 by Algorithm \ref{Full_Multigrid}, and $u_h^{\rm dir}$ and
 $\lambda_h^{\rm dir}$ denote the eigenfunction
 and eigenvalue approximation by direct eigenvalue solving
 (The left figure corresponds to the left mesh in Figure \ref{Initial_Mesh} and
the right figure corresponds to the right mesh in Figure \ref{Initial_Mesh}) }
\label{numerical_multi_grid_Model}
\end{figure}

From Figure \ref{numerical_multi_grid_Model}, we find the full multigrid scheme can obtain
the optimal error estimates as same as the direct eigenvalue problem
solving for the eigenvalue and the corresponding eigenfunction approximations.

We also check the convergence behavior for multi eigenvalue approximations with
Algorithm \ref{Full_Multigrid}. Here the first six eigenvalues
$\lambda=2\pi^2, 5\pi^2, 5\pi^2, 8\pi^2,$ $10\pi^2, 10\pi^2$
are investigated. We adopt the meshes in Figure \ref{Initial_Mesh} as
the initial meshes and the corresponding numerical results are shown in Figure
\ref{numerical_multi_grid_Model_Multi} which also exhibits the optimal convergence
of the full multigrid scheme.
\begin{figure}[ht]
\centering
\includegraphics[width=7cm,height=6cm]{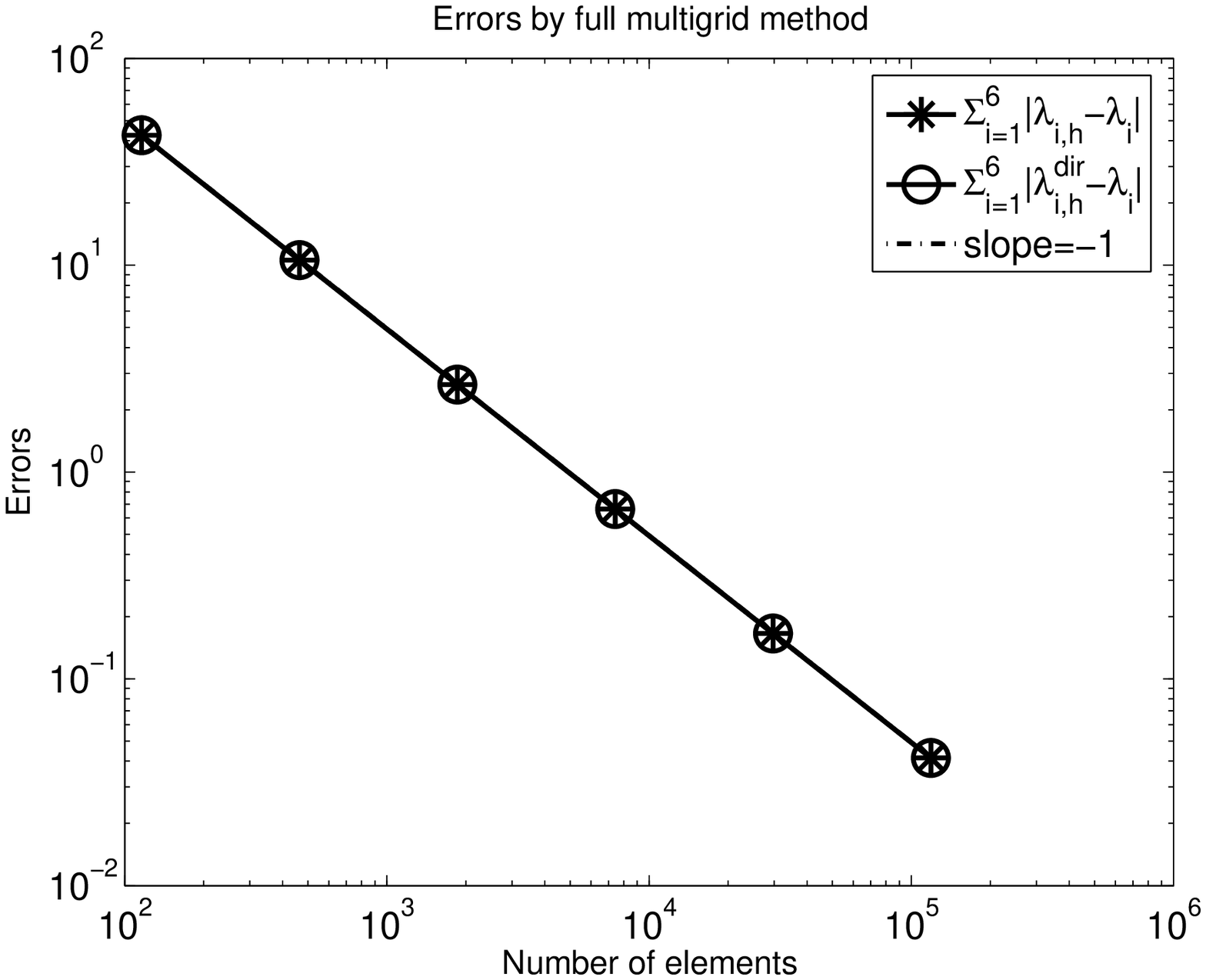}
\includegraphics[width=7cm,height=6cm]{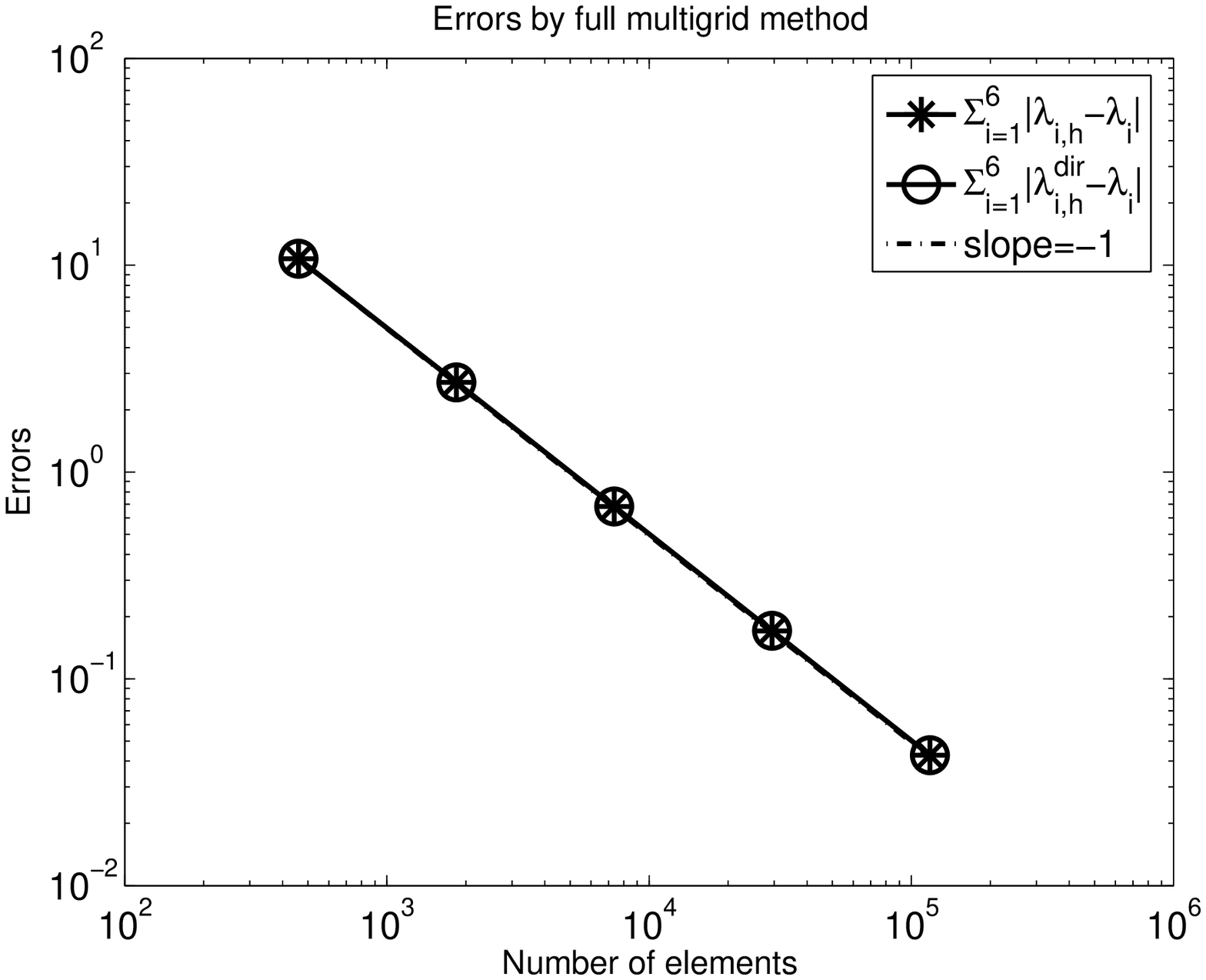}
\caption{\small\texttt The errors of the full multigrid
algorithm for the first six eigenvalues on the unit square,
 $\lambda_h$ denotes the eigenvalue approximation
 by Algorithm \ref{Full_Multigrid},
 $\lambda_h^{\rm dir}$ denotes the eigenvalue approximation by direct eigenvalue solving
 (The left figure corresponds to the left mesh in Figure \ref{Initial_Mesh} and
the right figure corresponds to the right mesh in Figure \ref{Initial_Mesh}) }
 \label{numerical_multi_grid_Model_Multi}
\end{figure}

\subsection{More general eigenvalue problem}
Here we give numerical results of the full multigrid method
for solving a more general eigenvalue problem on the unit square
domain $\Omega=(0, 1)\times (0, 1)$: Find $(\lambda,u)$ such that
\begin{equation}\label{Example_2}
\left\{
\begin{array}{rcl}
-\nabla\cdot\mathcal{A}\nabla u+\phi u&=&\lambda\rho u,\quad{\rm in}\ \Omega,\\
u&=&0,\quad\ \ \ {\rm on}\ \partial\Omega,\\
\int_{\Omega}\rho u^2d\Omega&=&1,
\end{array}
\right.
\end{equation}
where
\begin{equation*}
\mathcal{A}=\left (
\begin{array}{cc}
$$1+(x_1-\frac{1}{2})^2$$&$$(x_1-\frac{1}{2})(x_2-\frac{1}{2})$$\\
$$(x_1-\frac{1}{2})(x_2-\frac{1}{2})$$&$$1+(x_2-\frac{1}{2})^2$$
\end{array}
\right),
\end{equation*}
$\phi=e^{(x_1-\frac{1}{2})(x_2-\frac{1}{2})}$ and
$\rho=1+(x_1-\frac{1}{2})(x_2-\frac{1}{2})$.

In this example, we also use two coarse meshes which are shown in Figure \ref{Initial_Mesh}
as the initial meshes to investigate the convergence behaviors.
Since the exact solution is not known, we choose an adequately accurate eigenvalue
approximations with the extrapolation method (see, e.g., \cite{LinLin}) as the exact
eigenvalues to measure errors.
Figure \ref{numerical_multi_grid_General_Multi} gives the corresponding
numerical results for the first six eigenvalue approximations. In this example, we also choose
$m=2$, $p=2$ and $2$ conjugate gradient smoothing step in the presmoothing and postsmoothing procedure.
Here we also compare the numerical results with the direct algorithm. The corresponding results 
are shown in Figure \ref{numerical_multi_grid_General_Multi} which
  also exhibits the optimality of the error and complexity for
Algorithm \ref{Full_Multigrid}.
\begin{figure}[ht]
\centering
\includegraphics[width=7cm,height=6cm]{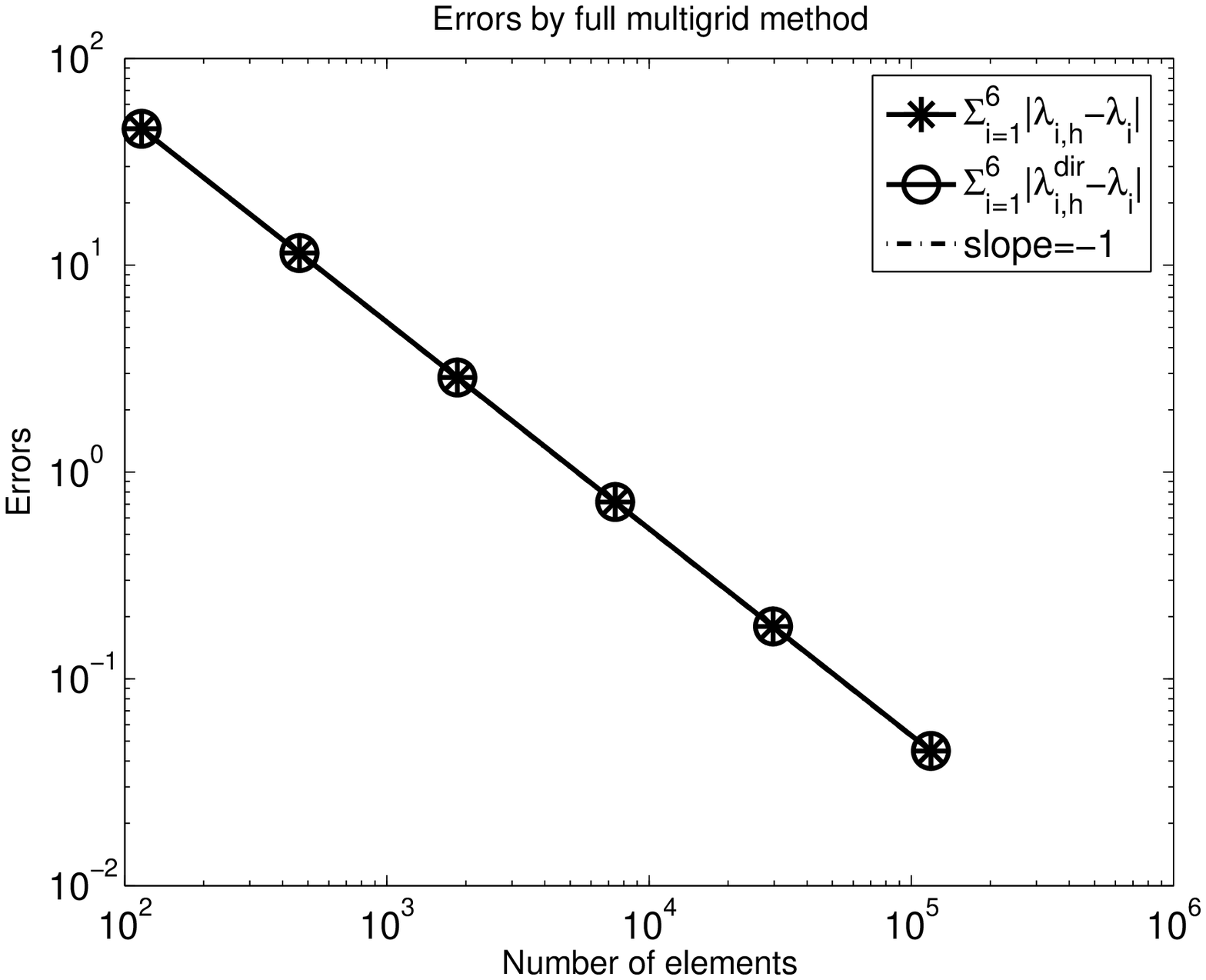}
\includegraphics[width=7cm,height=6cm]{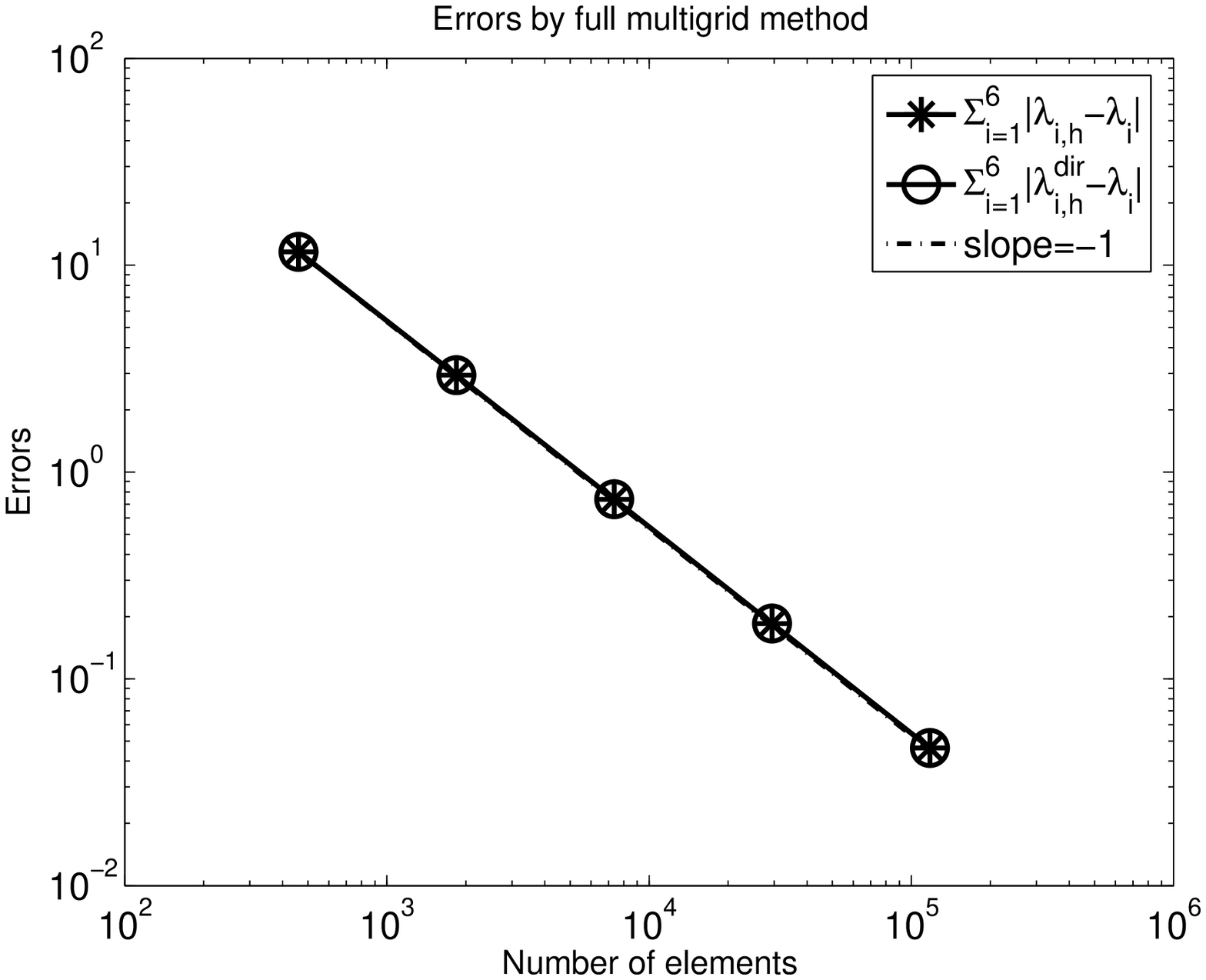}
\caption{\small\texttt The errors of the full multigrid
algorithm for the first six eigenvalues on the unit square,
 $\lambda_h$ denotes the eigenvalue approximation
 by Algorithm \ref{Full_Multigrid},
 $\lambda_h^{\rm dir}$ denotes the eigenvalue approximation by direct eigenvalue solving
 (The left figure corresponds to the left mesh in Figure \ref{Initial_Mesh} and
the right figure corresponds to the right mesh in Figure \ref{Initial_Mesh}) }
 \label{numerical_multi_grid_General_Multi}
\end{figure}

\section{Concluding remarks}
In this paper, we give a full multigrid scheme
to solve eigenvalue problems. The idea here is to use the multilevel correction method
to transform the solution of the eigenvalue problem to a series of solutions of the corresponding
boundary value problems, which can be solved by some multigrid iteration steps, and solutions of
eigenvalue problems defined on the coarsest finite element space.

We can replace the multigrid iteration by other types of efficient iteration schemes
such as algebraic multigrid method, the type of preconditioned schemes based on
the subspace decomposition and subspace corrections (see, e.g., \cite{BrennerScott, Xu}),
and the domain decomposition method (see, e.g., \cite{ToselliWidlund,XuZhou_Eigen_LocalParallel}).
The ideas can be extended to other types of linear and nonlinear eigenvalue
problems and other types problems. These will be investigated in  our future work.

\end{document}